\documentclass[graybox]{svmult}
\usepackage[margin=4.2cm]{geometry}

\usepackage{type1cm}        
\usepackage{makeidx}         
\usepackage{graphicx}        
\usepackage{multicol}        
\usepackage[bottom]{footmisc}

\usepackage{newtxtext}       %
\usepackage[varvw]{newtxmath}       
\usepackage{cite}

\usepackage{todonotes}

\usepackage{hyperref}
\usepackage{cleveref}
\usepackage{amsmath,amsfonts,amsthm}
\usepackage{mathtools}
\usepackage{commath}

\newcommand{\R}{\mathbb{R}}
\newcommand{\N}{\mathbb{N}}

\DeclareMathOperator*{\argmin}{arg\,min}

\newcommand{\rkhs}{\mathcal{H}}

\renewcommand{\b}[1]{\mathsf{#1}} 

\makeindex             

\begin{document}

\title*{Piecewise linear interpolation via kernels}
\author{Toni Karvonen, Gabriele Santin and Tizian Wenzel}
\institute{Toni Karvonen \at Lappeenranta--Lahti University of Technology LUT, Finland. \email{toni.karvonen@lut.fi}
\and Gabriele Santin \at Ca’ Foscari
University of Venice, Italy. \email{gabriele.santin@unive.it}
\and Tizian Wenzel \at Ludwig Maximilian University of Munich and Munich Center for Machine Learning, Germany. \email{wenzel@math.lmu.de}
}
%
%
\maketitle

\abstract*{
    We consider piecewise linear interpolation from the perspective of kernel interpolation and quadrature.
    If the Sobolev space $W_2^1(0, 1)$ is equipped with a suitable inner product, its reproducing kernel is piecewise linear and gives rise to piecewise linear interpolation. 
    We show that such kernels are Green kernels for certain second-order partial differential equations and use kernel-based superconvergence theory to obtain rates of convergence for approximation of functions lying in $W_2^s(0, 1)$ for $s \in [1, 2]$.
    The rates coincide with classical rates for linear splines.
    }
\abstract{
    We consider piecewise linear interpolation from the perspective of kernel interpolation and quadrature.
    If the Sobolev space $W_2^1(0, 1)$ is equipped with a suitable inner product, its reproducing kernel is piecewise linear and gives rise to piecewise linear interpolation. 
    We show that such kernels are Green kernels for certain second-order partial differential equations and use kernel-based superconvergence theory to obtain rates of convergence for approximation of functions lying in $W_2^s(0, 1)$ for $s \in [1, 2]$.
    The rates coincide with classical rates for linear splines.
    }
\section{Introduction}

Let $K \colon \Omega \times \Omega \to \R$ be a positive-semidefinite kernel on a set $\Omega$.
A Hilbert space of real-valued functions on $\Omega$ is a reproducing kernel Hilbert space (RKHS; \cite{Berlinet2004, Paulsen2016}) if and only if the point evaluation functional $f \mapsto f(x)$ is continuous for every $x \in \Omega$.
The classical Moore--Aronszajn theorem~\cite{Aronszajn1950,Paulsen2016} states that every $K$ induces a unique RKHS $\mathcal{H}$ in which the kernel is reproducing, which means that $\langle f, K(\cdot, x) \rangle_\mathcal{H} = f(x)$ for every
$f \in \mathcal{H}$ and $x \in \Omega$.
Conversely, every RKHS has a unique reproducing kernel that is positive-semidefinite.

Let $x_0, \ldots, x_n \in \Omega$ be pairwise distinct points and $f \colon \Omega \to \R$ a function.
Let $K_x = K(\cdot, x)$ denote a kernel translate.
Define
\begin{equation*}
    \b{K} = (K(x_i, x_j))_{i,j=0}^n \in \R^{(n+1) \times (n+1)} \quad \text{ and } \quad \b{f} = (f(x_i))_{i=0}^n \in \R^{n+1}.
\end{equation*}
If the Gram matrix $\b{K}$ is invertible (e.g., the kernel is positive-definite), the \emph{kernel interpolant}, $P_n f$, is the unique function $s$ in the span of the kernel translates $K_{x_0}, \ldots, K_{x_n}$ such that $s(x_i) = f(x_i)$ for every $i \in \{0, \ldots, n\}$.
The interpolant can be written as $P_n f = \sum_{i=0}^n c_i K_{x_i}$, where the coefficients $\b{c} = (c_0, \ldots, c_n)$ solve the linear system $\b{K} \b{c} = \b{f}$.
If $f \in \mathcal{H}$, the kernel interpolant is the orthogonal projection of $f$ onto the span of the kernel translates.
The orthogonal projection exists even if the Gram matrix $\b{K}$ is non-invertible.
The kernel interpolant is worst-case optimal in $\mathcal{H}$, which is to say that
\begin{equation*}
    (P_n f)(x) = \argmin_{u_0, \ldots, u_n \in \R} \sup_{\lVert f \rVert_\mathcal{H} \leq 1} \bigg\lvert f(x) - \sum_{i=0}^n u_i f(x_i) \bigg\rvert 
\end{equation*}
for every $x \in \Omega$.
We refer to \cite{wendland2005scattered} and Chapter~8 in~\cite{iske2019book} for information about kernel interpolation.
See Chapter~10 in~\cite{NovakWozniakowski2010} for the worst-case perspective.

Let $0 = x_0 < x_1 < \cdots < x_{n-1} < x_n = 1$ be a strictly increasing sequence of points on the unit interval and $\Delta_i = x_i - x_{i-1}$.
The \emph{piecewise linear interpolant} (or \emph{linear spline}) to a function $f \colon [0, 1] \to \R$ at these points is the function $L_n f \colon [0, 1] \to \R$ given by
\begin{equation*}
    (L_n f)(x) = f(x_{i-1}) + \frac{f(x_{i}) - f(x_{i-1})}{\Delta_i} (x - x_{i-1}) \quad \text{ for } \quad x \in [x_{i-1}, x_i],    
\end{equation*}
where $i \in \{1, \ldots, n\}$.
The purpose of this article is to study piecewise linear interpolation as kernel interpolation.
It is easy to establish a correspondence between the two types of interpolation if the kernel is, in our terminology, 2-piecewise linear.

\begin{definition}
    A symmetric kernel $K \colon [0, 1] \times [0, 1] \to \R$ is \emph{2-piecewise linear} if $K_x = K(\cdot, x)$ is affine on $[0, x]$ and $[x, 1]$ for every $x \in [0, 1]$.
\end{definition}

\begin{theorem} \label{thm:piecewise-linear}
    Let $0 = x_0 < x_1 < \cdots x_{n-1} < x_n = 1$.
    If $K$ is a 2-piecewise linear positive-semidefinite kernel on $[0, 1]$, then $P_n f = L_n f$ for every $f \in \mathcal{H}$.    
\end{theorem}
\begin{proof}    
    Let $f \in \mathcal{H}$.
    The kernel interpolant is $P_n f = \sum_{i=0}^n c_i K_{x_i}$ for constants $c_i$ such that $(P_n f)(x_i) = f(x_i)$ for every $i \in \{0, \ldots, n\}$.
    Every kernel translate in the sum is affine on $[x_{i-1}, x_{i}]$ for any $i \in \{1, \ldots, n\}$.
    Therefore $P_n f$ is affine on each of the intervals $[x_{i-1}, x_{i}]$.
    Because $P_n f$ interpolates $f \in \mathcal{H}$, it must therefore coincide with the piecewise linear interpolant $L_n f$.
\end{proof}

Although we have not found \Cref{thm:piecewise-linear} explicitly in the literature, the takeaway that piecewise linear interpolation (and, more generally, spline interpolation) can be recovered from kernel interpolation is well-known.
The connection of piecewise linear interpolation and the trapezoidal rule, an integral approximation obtained by integrating the interpolant, to optimal approximation and integration in Sobolev spaces and with the Brownian motion (i.e., via the kernel $K(x, y) = \min\{x, y\}$) has a history that stretches back to at least the 1950s~\cite{Diaconis1988, DucJacquet1973, Suldin1959, Suldin1960}. 
The work by Kimeldorf and Wahba~\cite{KimeldorfWahba1970, KimeldorfWahba1970b, KimeldorfWahba1971} from the 1970s that connects approximation in an RKHS and Bayesian estimation to splines is particularly famous.
See also \cite{Wahba1990} and Chapter~II in \cite{Ritter2000}.
However, to the best of our knowledge the tools developed for kernel-based approximation, in particular superconvergence, have not been previously used in a systematic study of piecewise linear approximation. 

Our contributions are (i) an identification of Sobolev spaces that give rise to piecewise linear interpolation and (ii) an interpretation of the resulting reproducing kernels as Green kernels and a subsequent application of superconvergence theory.
Specifically, in Section~\ref{sec:sobolev-space} we use the \Cref{thm:piecewise-linear} to cast piecewise linear interpolation as kernel interpolation when $\rkhs$ is the Sobolev space $W_2^1(0,1)$ of order one equipped with an inner product of the form
\begin{equation} 
    \langle f, g \rangle = \langle f', g' \rangle_{L_2(0, 1)} + \text{ suitable boundary terms}.
\end{equation}
As we show in Section~\ref{sec:green-kernels}, such kernels are Green kernels for certain second-order partial differential equations (PDEs).
This enables the application of \emph{kernel-based superconvergence theory}~\cite{karvonen2025general, schaback2018superconvergence, sloan2025doubling}, which we use in Section~\ref{sec:superconvergence} to obtain rates of convergence for linear interpolation when $f \in W_2^\theta(0, 1)$ for $\theta \in [1, 2] \setminus \{3/2\}$.
In the process we describe abstract interpolation spaces that are used in superconvergence theory.
The final convergence result is given in \Cref{cor:superconvergence-W1}.
These convergence rates are not new, coinciding with classical rates for linear splines~\cite{HedstromVarga1971,SwartzVarga1972} that we review in Section~\ref{sec:classical-bounds}.

\section{Piecewise linear kernels as reproducing kernels of $W_2^1(0,1)$} \label{sec:sobolev-space}

Let $p \in [1, \infty]$.
The Sobolev space $W^s_p(0, 1)$ of order $s \in \N_0$ on $[0, 1]$ consists of $s-1$ times differentiable functions such that the derivative $f^{(s-1)}$ is absolutely continuous and $f^{(s)}$, which exists almost everywhere, is in $L_p(0, 1)$.
The norm and seminorm are
\begin{equation*}
    \lVert f \rVert_{W^s_p(0,1)} = \bigg( \sum_{r = 0}^s \lVert f^{(r)} \rVert_{L_p(0,1)}^p \bigg)^{1/p} \quad \text{ and } \quad \lvert f \rvert_{W^s_p(0,1)} = \lVert f^{(s)} \rVert_{L_p(0,1)}.
\end{equation*}
Let $\alpha_0, \alpha_1, \alpha_2, \beta \in \R$. 
We consider the bilinear form 
\begin{equation} \label{eq:rkhs-inner-product}
    \langle f, g \rangle = \alpha_0 f(0) g(0) + \alpha_1 f(1) g(1) + \alpha_2 [f(0) g(1) + f(1) g(0)] + \beta \langle f', g' \rangle_{L_2(0, 1)}
\end{equation}
on $W_2^1(0, 1)$.
Note that $f(0)g(1)$ and $f(1)g(0)$ have the same prefactor $\alpha_2$ to ensure the symmetry of $\langle f, g \rangle$.
By introducing $\b{A} \in \R^{2 \times 2}$ and $\b{B} \colon W_2^1(0, 1) \to \R^2$ defined as
\begin{equation*}
    \b{A} = 
    \begin{pmatrix}
    \alpha_0 & \alpha_2 \\
    \alpha_2 & \alpha_1
    \end{pmatrix}
    \quad \text{ and }
    \quad
    \b{B}(f) = 
    \begin{bmatrix}
        f(0) \\ f(1) 
    \end{bmatrix},
\end{equation*}
we may write the bilinear form as $\langle f, g \rangle = \langle \b{B} (f), \b{A} \b{B} (g) \rangle_{\R^2} + \beta \langle f', g' \rangle_{L_2(0, 1)}$.
This means that $\langle \cdot, \cdot \rangle$ falls in the framework of~\cite{fasshauer2013reproducing}.
The Sobolev space $W_2^1(0,1)$ equipped with the bilinear form in~\eqref{eq:rkhs-inner-product} is an RKHS if $\beta$ is positive and $\b{A}$ positive-definite, which is equivalent to $\alpha_0 \alpha_1 - \alpha_2^2 > 0$ and $\alpha_0 + \alpha_1 > 0$.

\begin{theorem} \label{thm:sobolev-1-space}
    If $\beta > 0$ and $\b{A}$ is a positive-definite matrix or a diagonal matrix such that at least one of $\alpha_0, \alpha_1 \geq 0$ is positive, then the bilinear form in~\eqref{eq:rkhs-inner-product} is an inner product on \smash{$W_2^1(0, 1)$}.
    The resulting inner product space is an RKHS whose reproducing kernel is 2-piecewise linear and given by
    \begin{equation} \label{eq:W1-kernel}
        \begin{split}
        K(x, y) = \frac{1}{\beta \Delta} \big[ \alpha_1 \beta &+ \beta^2 - \max\{x,y\} (\alpha_1+\alpha_2) \beta  \\
        &+ \min\{x, y\} ( \alpha_0\alpha_1+\alpha_0\beta-\alpha_2^2+\alpha_2\beta ) + xy(\alpha_2^2 - \alpha_0\alpha_1) \big],
        \end{split}
    \end{equation}
    where $\Delta = \alpha_0 \alpha_1 - \alpha_2^2 + \beta(\alpha_0 + \alpha_1 + 2\alpha_2) > 0$.
\end{theorem}
\begin{proof}
    If $\b{A}$ is a positive-definite matrix (so that $\alpha_0, \alpha_1 > 0$) or a diagonal matrix such that at least one of $\alpha_0, \alpha_1 \geq 0$ is positive, there are $C, c > 0$ such that 
    \begin{equation} \label{eq:A-PD-implication}
        c ( \alpha_0 f(0)^2 + \alpha_1 f(1)^2) \leq \langle \b{B} (f), \b{A} \b{B} (f) \rangle_{\R^2} \leq C ( f(0)^2 + f(1)^2) .
    \end{equation}
    From \smash{$W_2^1(0, 1) \subset C(0,1)$} it thus follows that $\langle f, f \rangle = 0$ if and only if $f = 0$.
    Therefore~\eqref{eq:rkhs-inner-product} defines an inner product on \smash{$W_2^1(0, 1)$}.
    The Sobolev space \smash{$W^1_2(0, 1)$} equipped with the standard Sobolev inner product is a Hilbert space. 
    The fundamental theorem of calculus and the Cauchy--Schwarz inequality give \smash{$f(x) - f(y) = \int_y^x f'(z) \dif z \leq \lVert f' \rVert_{L_2(0,1)}$}.
    Integration over $y$ and another application of Cauchy--Schwarz yield
    \begin{equation*}
        \lvert f(x) \rvert \leq \lvert f(y) \rvert + \lvert f(x) - f(y) \rvert \leq \lVert f \rVert_{L_2(0,1)} + \lVert f' \rVert_{L_2(0,1)} \leq \sqrt{2} \, \lVert f \rVert_{W^1_2(0,1)} .
    \end{equation*}
    That is, point evaluations are bounded and $W^1_2(0,1)$ is thus an RKHS.
    In combination with~\eqref{eq:A-PD-implication} and estimates like
    \begin{equation*}
        \begin{split}
        \lVert f \rVert_{L_2(0,1)}^2 = \int_0^1 \bigg( \int_0^x f'(y) \dif y - f(0) \bigg)^2 \dif x &\leq 2 \int_0^1 \bigg( \int_0^x f'(y)^2 \dif y + f(0)^2 \bigg) \dif x \\
        &\leq 2 [\lVert f' \rVert_{L_2(0,1)}^2  + f(0)^2 ],
        \end{split}
    \end{equation*}
    this shows that the Sobolev norm is equivalent to the norm induced by~\eqref{eq:rkhs-inner-product}.
    Therefore $W^1_2(0, 1)$ equipped with the inner product~\eqref{eq:rkhs-inner-product} is an RKHS.

    Let $K$ be the unique reproducing kernel.
    For the constant function $h \equiv 1 \in W^1_2(0,1)$ the reproducing property and~\eqref{eq:rkhs-inner-product} yield 
    \begin{equation} \label{eq:eq11}
        1 = \langle h, K_x \rangle = (\alpha_0 + \alpha_2) K_x(0) + (\alpha_1 + \alpha_2) K_x(1).
    \end{equation}
    Let $y \in (0, 1)$.
    The functions
    \begin{equation*}
        f_y(x) =
        \begin{cases}
            0 &\text{ if } \: x \leq y, \\
            x-y & \text{ if } \: x > y             
        \end{cases}
        \quad
        \text{ and }
        \quad
        g_y(x) = 
        \begin{cases}
            y-x &\text{ if } \: x \leq y, \\
            0 & \text{ if } \: x > y             
        \end{cases}
    \end{equation*}
    are in $W^1_2(0, 1)$.
    For $x \leq y$ the reproducing property yields
    \begin{equation} \label{eq:eq12}
        \begin{split}
        0 = f_y(x) = \langle f_y, K_x \rangle &= \alpha_1 (1 - y) K_x(1) + \alpha_2(1 - y) K_x(0) + \beta \int_y^1 K_x'(z) \dif z \\
        &= \alpha_1 (1 - y) K_x(1) + \alpha_2(1 - y) K_x(0) + \beta [ K_x(1) - K_x(y) ]
        \end{split}
    \end{equation}
    and 
    \begin{equation} \label{eq:eq13}
        \begin{split}
        y - x = g_y(x) = \langle g_y, K_x \rangle &= \alpha_0 y K_x(0) + \alpha_2 y K_x(1) - \beta \int_0^y K_x'(z) \dif z \\
        &= \alpha_0 y K_x(0) + \alpha_2 y K_x(1) + \beta [K_x(0) - K_x(y) ] .
        \end{split}
    \end{equation}
    It is straightforward to solve $K_x(y)$ from equations \eqref{eq:eq11}--\eqref{eq:eq13} for $x \leq y$.
    Doing so gives
    \begin{equation*}\label{eq:left_kernel}
        K_x(y) = \frac{ \alpha_1 \beta + \beta^2 - y (\alpha_1+\alpha_2) \beta + x ( \alpha_0\alpha_1+\alpha_0\beta-\alpha_2^2+\alpha_2\beta ) + xy(\alpha_2^2 - \alpha_0\alpha_1) }{ \beta \Delta } ,
    \end{equation*}
    where $\Delta = \alpha_0 \alpha_1 - \alpha_2^2 + \beta(\alpha_0 + \alpha_1 + 2\alpha_2)$ is positive by the assumption on $\b{A}$.
    If $x > y$, Equations \eqref{eq:eq12} and~\eqref{eq:eq13} become 
    \begin{equation} \label{eq:eq22}
        x - y = f_y(x) = \alpha_1 (1 - y) K_x(1) + \alpha_2 (1 - y) K_x(0) + \beta [ K_x(1) - K_x(y) ]
    \end{equation}
    and
    \begin{equation} \label{eq:eq23}
        0 = g_y(x) = \alpha_0 y K_x(0) + \alpha_2 y K_x(1) + \beta [K_x(0) - K_x(y) ].
    \end{equation}
    Solving $K_x(y)$ from~\eqref{eq:eq11}, \eqref{eq:eq22}, and~\eqref{eq:eq23} yields
    \begin{equation*}\label{eq:right_kernel}
        K_x(y) = \frac{ \alpha_1\beta+\beta^2 - x (\alpha_1 + \alpha_2)\beta + y ( \alpha_0\alpha_1+\alpha_0\beta-\alpha_2^2+\alpha_2\beta ) + xy (\alpha_2^2 
-\alpha_0\alpha_1)}
{\beta \Delta}
    \end{equation*}
    for $x > y$. 
    This is the claimed form of the kernel, which is clearly 2-piecewise linear.
\end{proof}

\subsection{Examples of piecewise linear kernels covered by \Cref{thm:sobolev-1-space}} \label{sec:examples-1}

For $\alpha_0 > 0$ and $\alpha_1 = \alpha_2 = 0$
we get the \emph{released Brownian motion kernel} and for $\alpha_1 > 0$ and $\alpha_0 = \alpha_2 = 0$ the \emph{released reverse Brownian motion kernel}.
These kernels are
\begin{equation} 
        K(x, y) = \alpha_0^{-1} + \beta^{-1} \min\{x, y\} \quad \text{ and } \quad K(x, y) = \alpha_1^{-1} + \beta^{-1} ( 1 - \max\{x, y\}) ,
\end{equation}
respectively.
These kernels are known to be positive-definite.

\subsection{Examples of piecewise linear kernels as limits of \Cref{thm:sobolev-1-space}} \label{sec:examples-2}

Consider the bilinear form $\langle f, g \rangle = \beta \langle f', g' \rangle_{L_2(0,1)}$ that \Cref{thm:sobolev-1-space} does not cover.
This is not an inner product on \smash{$W_2^1(0, 1)$}, as is easily seen by considering any non-zero constant function.
However, $\langle \cdot, \cdot \rangle$ becomes an inner product under zero Dirichlet boundary conditions.
There are three cases, each of which is easily verified with a technique similar to the proof of \Cref{thm:sobolev-1-space}:
\begin{enumerate}
    \item[(a)] The space $\{ f \in W_2^1(0, 1) : f(0) = 0\}$ equipped with $\langle \cdot, \cdot \rangle$ is an RKHS whose reproducing kernel is the \emph{Brownian motion kernel} $K(x, y) = \beta^{-1} \min\{x, y\}$.
\item[(b)] The space $\{ f \in W_2^1(0, 1) : f(1) = 0\}$ equipped with $\langle \cdot, \cdot \rangle$ is an RKHS whose reproducing kernel is the \emph{reverse Brownian motion kernel} $K(x, y) = \beta^{-1} (1 - \max\{x, y\})$.
\item[(c)] The space $\{ f \in W_2^1(0, 1) : f(0) = f(1) = 0\}$ equipped with $\langle \cdot, \cdot \rangle$ is an RKHS whose reproducing kernel is the \emph{Brownian bridge kernel} $K(x, y) = \beta^{-1} ( \min\{x, y\} - xy)$.
\end{enumerate}
Each of these kernels is positive-semidefinite and 2-piecewise linear.
Formally, these three kernels are obtained from~\eqref{eq:W1-kernel} by setting (a) $\alpha_1 = \alpha_2 = 0$ and $\alpha_0 \to \infty$, (b) $\alpha_0 = \alpha_2 = 0$ and $\alpha_1 \to \infty$, and (c) $\alpha_0 = \alpha_1 \to \infty$.

\begin{remark}
    These examples show that the assumption $f \in \mathcal{H}$ in Theorem~\ref{thm:piecewise-linear} is not superfluous.
    For example, in case~(a) every kernel interpolant evaluates to zero at the origin and cannot therefore interpolate functions such that $f(0) \neq 0$.
\end{remark}

The matrix $\b{A}$ is positive-semidefinite if $\alpha_0 \alpha_1 - \alpha_2^2 = 0$.
While \Cref{thm:sobolev-1-space} is not applicable in this case, selecting such $\alpha_0$, $\alpha_1$, and $\alpha_2$ gives the kernel
\begin{equation*}
    K(x, y) = \frac{\alpha_1 + \beta - \max\{x, y\} (\alpha_1 + \alpha_2) + \min\{x, y\} (\alpha_0 + \alpha_2)}{\beta(\alpha_0 + \alpha_1 + 2\alpha_2)}
\end{equation*}
provided that $\alpha_0 + \alpha_1 + 2\alpha_2 > 0$.
Furthermore, if $\gamma \coloneqq \alpha_1 + \alpha_2 = \alpha_0 + \alpha_2$, we can use the identity $\lvert x - y\rvert = \max\{x, y\} - \min\{x, y\}$ to 
simplify the kernel to
\begin{equation*}
    K(x, y) = (2\beta\gamma)^{-1}(\alpha_1 + \beta - \gamma \lvert x - y \rvert) .
\end{equation*}
Let $\varepsilon > 0$. 
This kernel is clearly related to the positive-semidefinite \emph{Wendland kernel} $K(x, y) = \max\{0, 1 - \varepsilon \lvert x - y \rvert \}$.
Note that the Wendland kernel is 2-piecewise linear on $[0, 1]$ if and only if $\varepsilon \in (0, 1]$.

\section{Piecewise linear kernels as Green kernels} \label{sec:green-kernels}
In order to apply superconvergence results~\cite{karvonen2025general}, it will be crucial to associate the kernel~\eqref{eq:W1-kernel} to a boundary value problem by providing a connection with the corresponding Green kernel.
Recall that the Green kernel of a PDE $\mathcal{L} u = f$ on some domain $\Omega$ subject to some boundary conditions is a function $G \colon \Omega \times \Omega \to \R$ such that $u(x) = \int_\Omega G(x, y) f(y) \dif y$.
The connection between Green kernels and reproducing kernels has been studied in~\cite{fasshauer2011reproducing, 
fasshauer2013reproducing}.
The following result shows that the kernels of Theorem~\ref{thm:sobolev-1-space} are Green kernels of certain simple second-order PDEs.
\begin{corollary}\label{cor:green}
Under the assumptions of~\Cref{thm:sobolev-1-space}, the kernel $K$ in~\eqref{eq:W1-kernel} is the Green kernel of the PDE
\begin{align}
\begin{aligned} 
\label{eq:pde_green}
-\beta u''(x) &=f(x), \quad x\in(0,1), \\
\beta u'(0) &= \alpha_0 u(0) + \alpha_2 u(1), \\
\beta u'(1) &= -\alpha_1 u(1) - \alpha_2 u(0).
\end{aligned}
\end{align}
\end{corollary}

\begin{proof}
Let $G$ be the Green kernel of~\eqref{eq:pde_green} and let $G_x(y)\coloneqq G(x, y)$.
Following Sections 6.2--6.4 in~\cite{fasshauer2015kernel}, $G$ is uniquely defined by the following properties for all $x\in(0,1)$: (i) $G_x$ is affine on $[0,x]$ and $[x,1]$, and continuous on $[0,1]$; (ii) $\beta(\lim_{y\to x^-}G_x'(y)-\lim_{y\to x^+}G_x'(y))=1$; and (iii) $G_x$ satisfies the boundary conditions of~\eqref{eq:pde_green}.
We show that the kernel $K$ of Theorem~\ref{thm:sobolev-1-space} satisfies these properties, so that $K=G$ by uniqueness.
The first property is clearly satisfied. To verify the second one we use~\eqref{eq:W1-kernel} to obtain
\begin{equation*} 
\beta \Delta K_x'(y) 
= y(\alpha_2^2 - \alpha_0\alpha_1)+
\begin{cases}
\alpha_0\alpha_1+\alpha_0\beta-\alpha_2^2+\alpha_2\beta &\text{ if } \: y \leq x, \\
-  (\alpha_1+\alpha_2)\beta  &\text{ if } \: y > x, 
\end{cases}
\end{equation*}
so that, using the definition of $\Delta$ in \Cref{thm:sobolev-1-space}, we get
\begin{equation*}
\beta\Delta\big(\lim_{y\to x^-}K_x'(y)-\lim_{y\to x^+}K_x'(y)\big)
= (\alpha_0\alpha_1+\alpha_0\beta-\alpha_2^2+\alpha_2\beta )+ (\alpha_1 + \alpha_2)\beta
= \Delta,
\end{equation*}
which gives Property (ii) by simplifying $\Delta$.
For the boundary conditions we compute
\begin{align*}
\beta\Delta&\left(\alpha_0 G_x(0) + \alpha_2 G_x(1)\right)\\
&=\alpha_0 \left((\alpha_1+\beta)\beta  - x (\alpha_1 + \alpha_2)\beta\right)\\
&\phantom{=}+\alpha_2 \big((\alpha_1+\beta)\beta + x(\alpha_2^2 - \alpha_0\alpha_1) - (\alpha_1+\alpha_2) \beta + x 
(\alpha_0\alpha_1+\alpha_0\beta-\alpha_2^2+\alpha_2\beta )\big)\\
&=x(\alpha_2^2-\alpha_0\alpha_1)\beta + (\alpha_0\alpha_1+ \alpha_0\beta+\alpha_2\beta-\alpha_2^2)\beta
\end{align*}
and
\begin{equation*}
\beta\Delta( \beta G_x'(0))
=x\beta(\alpha_2^2 - \alpha_0\alpha_1)+ (\alpha_0\alpha_1+\alpha_0\beta-\alpha_2^2+\alpha_2\beta)\beta,
\end{equation*}
showing that the left boundary condition is satisfied. The proof for the right boundary condition is similar, proving that $K=G$.
\end{proof}

\section{Convergence theory} \label{sec:superconvergence}

Superconvergence theory for kernel-based approximation allows deducing error estimates for functions in certain subspaces of $\rkhs$ from error estimates that hold for all $f \in \rkhs$~\cite{karvonen2025general}. We apply it to the kernel~\eqref{eq:W1-kernel} by leveraging its connection with Green kernels provided in Section~\ref{sec:green-kernels}.

Consider the integral operator $T$ given by
\begin{equation*}
    T f = \int_0^1 K(\cdot, y) f(y) \dif y  \quad \text{ for } \quad f \in L_2(0, 1).
\end{equation*}
Superconvergence theory provides error estimates for functions in the \emph{interpolation space} (defined via the $K$-method of real interpolation; see~\cite{triebel1978interpolation, karvonen2025general} for details)
\begin{equation*}
    \rkhs_\theta = (\rkhs, T L_2(0,1))_{\theta - 1, 2} \text{ for } \quad \theta \in [1, 2]
\end{equation*}
between the RKHS and the image of $L_2(0,1)$ under $T$.
We identify $\rkhs_1 = \rkhs$ and $\rkhs_2 = TL_2(0,1)$.
Interpolation spaces can be alternatively defined as \emph{power spaces} of the RKHS.
The following general superconvergence theorem follows from~\cite[Cor.\@~15]{karvonen2025general}.

\begin{theorem} \label{thm:superconvergence-general}
    If $\lVert f - P_n f \rVert_{L_2(0, 1)} \leq \varepsilon \lVert f \rVert_\rkhs$ for some $\varepsilon > 0$ and all $f \in \rkhs$, then
    \begin{equation*}
        \lVert u - P_n u \rVert_{L_2(0, 1)} \leq \varepsilon^{\theta} \lVert u \rVert_{\rkhs_{\theta}} 
    \end{equation*}
    for any $\theta \in [1, 2]$ and all $u \in \rkhs_{\theta}$.
\end{theorem}

To use \Cref{thm:superconvergence-general} we need an estimate on interpolation.
Recall that $\Delta_i = x_i - x_{i-1}$.
Let $h = \max\{ \Delta_i \}_{i=1}^{n}$ denote the fill-distance.

\begin{proposition} \label{prop:sampling-ineq}
    Let $p \in [1, \infty)$.
    If $f - P_n f \in W_p^1(0, 1)$, then 
    \begin{equation*}
        \lVert f - P_n f \rVert_{L_p(0,1)} \leq \frac{1}{p^{1/p}} \cdot h  \cdot \lvert f \rvert_{W_p^1(0, 1)} .
    \end{equation*}    
\end{proposition}
\begin{proof}
    Let $u \in W_p^1(0, 1)$ be any function that vanishes at $a = x_{i-1}$ and $b = x_i$ and let $p' = p/(p-1)$ be the Hölder conjugate of $p$.
    Then
    \begin{equation*}
        \begin{split}
        \lVert u \rVert_{L_p(a,b)}^p 
        = \int_a^b \bigg\lvert \int_a^x u'(z) \dif z \bigg\rvert^p \dif x &\leq \int_a^b \bigg( \int_a^x 1 \dif z \bigg)^{p-1} \bigg( \int_a^x \lvert u'(z) \rvert^p \dif z \bigg) \dif x \\
        &\leq \lvert u \rvert_{W_p^1(a,b)}^p \int_a^b (x - a)^{p-1} \dif x \\
        &\leq \frac{h^{p}}{p} \lvert u \rvert_{W_p^1(a,b)}^p 
        \end{split}
    \end{equation*}
    by the fundamental theorem of calculus and Hölder's inequality.
    The claim follows from
    \begin{equation*}
        \lVert u \rVert_{L_p(0,1)}^p \leq \frac{h^p}{p} \sum_{i=1}^n \lvert u \rvert_{W_p^1(x_{i-1},x_i)}^p = \frac{h^p}{p} \lvert u \rvert_{W_p^1(0, 1)}^p. \qedhere
    \end{equation*}
\end{proof}

\Cref{prop:sampling-ineq} is an example of a sampling inequality, much more general versions of which can be found in~\cite{arcangeli2007sampling, wendland2005sampling}.
\Cref{thm:superconvergence-general} and \Cref{prop:sampling-ineq} yield the following general superconvergence result for any kernel whose RKHS embeds continuously in $W_2^1(0,1)$.

\begin{corollary} \label{cor:superconvergence-sobolev-1}
    Let $K$ be any positive-semidefinite kernel on $[0, 1]$ such that $\rkhs \subset W_2^1(0,1)$ and $\lvert f \rvert_{W_2^1(0,1)} \leq c \lVert f \rVert_\rkhs$ for some $c \geq 0$ and all $f \in \rkhs$.
    Then
    \begin{equation*}
        \lVert u - P_n u \rVert_{L_2(0,1)} \leq \bigg( \frac{c}{\sqrt{2}} \bigg)^{\theta} h^{\theta} \lVert u \rVert_{\rkhs_{\theta}}
    \end{equation*}
    for any $\theta \in [1, 2]$ and all $u \in \rkhs_{\theta}$.
\end{corollary}

As $\langle \b{B}(f), \b{A}\b{B}(f) \rangle_{\R^2} \geq 0$ if the matrix $\b{A}$ is positive-semidefinite, \Cref{cor:superconvergence-sobolev-1} applies to the kernel in~\eqref{eq:W1-kernel} with $c = 1/\sqrt{\beta}$.

Because the power spaces $\rkhs_{\theta}$ are rather abstract, it is desirable to describe them in more concrete terms. Using \Cref{cor:green} we can completely characterize them as follows.
Note that in the proposition we are omitting the case $\theta=3/2$, which is critical because it is the threshold value for the regularity of the trace operator. This implies that the boundary conditions are not simply active or inactive, but need instead to be enforced in some weak sense. We refer to Theorem 1 in~\cite[Sec.\@~4.3.3]{triebel1978interpolation} and Remark 4 in~\cite[Section 4.3.3]{triebel1978interpolation} for further details.

\begin{proposition}
Under the assumptions of~\Cref{thm:piecewise-linear}, the power space $\rkhs_\theta$ for the kernel $K$ in~\eqref{eq:W1-kernel} is
\begin{equation*}
\mathcal H_\theta
= 
\begin{cases}
W_2^{\theta}(0,1) &\text{ if } \: \theta \in[1,3/2), \\
\{u\in W_2^{\theta}(0,1): u~ \text{satisfies the boundary conditions \eqref{eq:pde_green}}\} &\text{ if } \: \theta \in( 3/2,2],
\end{cases}
\end{equation*}
where equality is in the sense of norm equivalence.
\end{proposition}
\begin{proof}
The statement is clear for $\theta=1$. For $\theta=2$ we use~\Cref{cor:green} and observe that for any $f\in L_2(\Omega)$ the function \smash{$u = \int_0^1 K(\cdot,y) f(y) \dif y = T f$} solves the PDE in~\eqref{eq:pde_green} with right-hand side $f$, implying that $\rkhs_2=TL_2(0, 1)$ is the set of solutions of~\eqref{eq:pde_green}. More precisely, letting $B(u)$ be the boundary operator encoding the boundary conditions of~\eqref{eq:pde_green}, we have
\begin{equation*}
\rkhs_2
= W_{2,B}^\tau(0,1)\coloneqq \{ u \in W_2^2(0, 1) : Bu=0\}.
\end{equation*}
For $\theta\in(1,2)$ we have $\rkhs_\theta = (L_2(0,1), \rkhs_2)_{\theta/2, 2}$ with equivalent norms~\cite{karvonen2025general}, and to characterize this interpolation space we use a standard argument on the interpolation of the domain of elliptic operators.
Specifically, Theorem 1 in~\cite[Sec.\@~4.3.3]{triebel1978interpolation} ensures that for $0<\sigma<1$ we have
\begin{equation*}
(L_2(0,1), W_{2,B}^\tau(0,1))_{\sigma, 2}
=
\begin{cases}
W_2^{2\sigma}(0, 1) &\text{if } \sigma < 3/4,\\    
\{u\in W_2^{2\sigma}(0, 1): B(u)=0\} &\text{if } \sigma> 3/4, 
\end{cases}
\end{equation*}
concluding the proof.
\end{proof}

Having characterized $\rkhs_{\theta}$ for kernels from \Cref{thm:sobolev-1-space}, we obtain the following general superconvergence statement. 

\begin{corollary} \label{cor:superconvergence-W1}
Suppose that the assumptions of \Cref{thm:sobolev-1-space} hold. 
Let $K$ be the kernel in~\eqref{eq:W1-kernel} and $\theta \in [1,2] \setminus \{3/2\}$.
Then
\begin{equation*}
    \lVert u - P_n u \rVert_{L_2(0,1)} \leq \bigg( \frac{1}{\sqrt{2 \beta}} \bigg)^{\theta} h^{\theta} \lVert u \rVert_{\rkhs_{\theta}}
\end{equation*}
for all $u \in W_2^{\theta}(0,1)$ that satisfy the boundary conditions~\eqref{eq:pde_green} if $\theta > 3/2$ and for all $u \in W_2^{\theta}(0, 1)$ if $\theta < 3/2$.
\end{corollary}

Since the kernel in~\eqref{eq:W1-kernel} is 2-piecewise linear and its kernel interpolants hence piecewise linear interpolants by \Cref{thm:piecewise-linear}, \Cref{cor:superconvergence-W1} is a variant of classical convergence results for linear splines.
These results are briefly reviewed in the next section.

\section{Existing convergence theory} \label{sec:classical-bounds}

This section collects existing error bounds for piecewise linear interpolation and the trapezoidal rule, which is an integral approximation obtained by integrating the interpolant.
The theorems collected here should be compared to those in \Cref{sec:superconvergence}.
Here $C^{s,\alpha}(0, 1)$ is the space of functions that are $s$ times continuously differentiable and whose $s$th derivative is $\alpha$-Hölder continuous.
The Hölder seminorm is
\begin{equation*}
    \lvert f \rvert_{C^{s,\alpha}(0,1)} = \sup_{x \neq y} \frac{\lvert f^{(s)}(x) - f^{(s)}(y) \rvert}{\lvert x - y \rvert^\alpha} .
\end{equation*}
Let $p, q \in [1, \infty]$ and $\theta \in (0, 1)$. The Besov space $B_{p,q}^\sigma(0, 1)$ can be obtained from interpolation of Sobolev spaces (recall \Cref{sec:superconvergence}):
\begin{equation*}
    B_{p,q}^\sigma(0, 1) = (W_p^{s_0}(0,1), W_p^{s_1}(0,1))_{\theta,q} \quad \text{ for } \quad \sigma = \theta s_1 + (1 - \theta) s_0
\end{equation*}
if $s_0 \neq s_1$, where the norms are equivalent.
Besov spaces coincide with fractional Sobolev spaces when $p = q$.
That is, $B_{p,p}^\sigma(0, 1) = W_p^\sigma(0, 1)$.
See~\cite{triebel1978interpolation} for full theory of Besov and Sobolev spaces.
Recall that $\Delta_i = x_i - x_{i-1}$ and $h = \max\{ \Delta_i \}_{i=1}^n$.
We use $\lesssim$ to denote an inequality that holds up a multiplicative constant independent of $h$ and~$f$.

\subsection{Piecewise linear interpolation}

The following theorem is a combination of Corollary~3.3 for Parts~(b) and~(f), Theorem~3.5 from \cite{SwartzVarga1972} for Parts~(a) and~(c), and results in~\cite{HedstromVarga1971} (see Theorem~4.6 and p.\@~314) for Parts~(d) and~(e).
We use the notation $(x)_+ = \max\{x, 0\}$.

\begin{theorem} \label{thm:spline-rates}
    Let $p,q \in [1, \infty]$ and $r \in [2, \infty]$.
    \begin{enumerate}
        \item[(a)] If $f \in C^{0,\alpha}(0, 1)$ for $\alpha \in (0, 1]$, then
        \begin{equation*}
            \lVert f - L_n f \rVert_{L_r(0, 1)} \lesssim h^{\alpha - (1/2 - 1/r)_+} \lvert f \rvert_{C^{0,\alpha}(0, 1)}.
        \end{equation*}
        \item[(b)] If $f \in W^1_p(0, 1)$, then
        \begin{equation*}
            \lVert f - L_n f \rVert_{L_r(0, 1)} \lesssim h^{1 - (1/p - 1/r)_+} \lvert f \rvert_{W^1_p(0, 1)}.
        \end{equation*}                
        \item[(c)] If $f \in C^{1,\alpha}(0, 1)$ for $\alpha \in (0, 1]$, then
        \begin{equation*}
            \lVert f - L_n f \rVert_{L_r(0, 1)} \lesssim h^{(1+\alpha) - (1/2 - 1/r)_+} \lvert f \rvert_{C^{1,\alpha}(0, 1)}.
        \end{equation*}
        \item[(d)] If $f \in B_{p,q}^\sigma(0, 1)$ for $\sigma \in (1, 2)$, then
        \begin{equation*}
            \lVert f - L_n f \rVert_{L_r(0, 1)} \lesssim h^{\sigma - (1/p - 1/r)_+} \lVert f \rVert_{B_{p,q}^\sigma(0,1)} .
        \end{equation*}
        \item[(e)] In particular, if $f \in W_{p}^\sigma(0, 1)$ for $\sigma \in (1, 2)$, then
        \begin{equation*}
            \lVert f - L_n f \rVert_{L_r(0, 1)} \lesssim h^{\sigma - (1/p - 1/r)_+} \lVert f \rVert_{W_{p}^\sigma(0,1)} .
        \end{equation*}
        \item[(f)] If $f \in W^2_p(0, 1)$, then
        \begin{equation*}
             \lVert f - L_n f \rVert_{L_r(0, 1)} \lesssim h^{2 - (1/p - 1/r)_+ } \lvert f \rvert_{W^2_p(0,1)}.
         \end{equation*}
    \end{enumerate}
\end{theorem}

\subsection{Trapezoidal rule}

The \emph{trapezoidal rule}, $T_n(f)$, is an approximation of the integral $I(f) = \int_0^1 f(x) \dif x$ that is obtained by integrating the piecewise linear interpolant:
\begin{equation*}
  \begin{split}
  T_n(f) = \int_0^1 (L_n f)(x) \dif x = \sum_{i=1}^n \frac{f(x_i) + f(x_{i-1})}{2} \Delta_i .
  \end{split}
\end{equation*}
The \emph{kernel quadrature rule} is the quadrature that is obtained by integrating the kernel interpolant~\cite{santin2021sampling}. 
Because 2-piecewise linear kernels have piecewise linear interpolants by \Cref{thm:piecewise-linear}, kernel quadrature for such kernels coincides with the trapezoidal rule.

A combination of Theorems~1.1, 1.8, and 1.19 in~\cite{CruzUribe2002} [Parts~(a), (b) and (d)] and Theorem~4 in~\cite{DragomirMabizela1999} [Part~(c)] gives the following theorem on the convergence of the trapezoidal rule in various function spaces.

\begin{theorem} \label{thm:cruz-uribe}
    Suppose that $x_i = i/n$ for $i \in \{0, \ldots, n\}$.
    Define $f_r(x) = f(x) - rx$ for a given function $f$ and $r \in \R$.
    \begin{enumerate}
        \item[(a)] If $f \in C^{0,\alpha}(0, 1)$ for $\alpha \in (0, 1]$, then
        \begin{equation*} 
            \lvert I(f) - T_n(f) \rvert \leq \frac{1}{(1+\alpha) 2^\alpha} \cdot n^{-\alpha} \inf_{r \in \R} \, \lvert f_r \rvert_{C^{0,\alpha}(0,1)} .
        \end{equation*}
        \item[(b)] If $f \in W^1_p(0, 1)$ for $p \in [1, \infty]$ with Hölder conjugate $p' = p/(p-1)$, then
        \begin{equation*} 
            \lvert I(f) - T_n(f) \rvert \leq \frac{1}{2(p'+1)^{1/p'}} \cdot n^{-1} \inf_{r \in \R} \, \lvert f_r \rvert_{W_p^1(0, 1)} .
        \end{equation*}
        \item[(c)] If $f \in C^{1,\alpha}(0, 1)$ for $\alpha \in (0, 1]$, then 
        \begin{equation*}
            \lvert I(f) - T_n(f) \rvert = O(n^{-(1 + \alpha)}).
        \end{equation*}
        \item[(d)] If $f \in W^2_p(0, 1)$ for $p \in [1, \infty]$ with Hölder conjugate $p' = p/(p-1)$, then
        \begin{equation*}
            \lvert I(f) - T_n(f) \rvert \leq \frac{\mathrm{B}(p'+1, p'+1)}{2} \cdot n^{-2} \lvert f \rvert_{W_p^2(0, 1)} ,
        \end{equation*}
        where $\mathrm{B}$ is the beta function.
    \end{enumerate}
    The bounds in Parts~(a), (b) and~(d) are sharp, in that in each case there is $f$ in the appropriate function space such that an equality holds.
\end{theorem}

\section{Discussion}

This article studied piecewise linear interpolation in the framework of kernel interpolation.
We believe that many results in this article could be taken much further:
\begin{itemize}
    \item The inner product~\eqref{eq:rkhs-inner-product} does not contain the term $\langle f, g \rangle_{L_2(0,1)}$. 
    Inner products that include this term correspond to reproducing kernels written in terms of $\exp$; see Examples~13, 14, and 17 in~\cite[Sec.\@~7.4]{Berlinet2004}.
    For example, the full Sobolev inner product $\langle f, g \rangle = \langle f, g \rangle_{L_2(0,1)} + \langle f', g' \rangle_{L_2(0,1)}$ gives rise to the reproducing kernel
    \begin{equation*}
        K(x, y) = \frac{\cosh(\min\{x,y\}) \cosh(1-\max\{x,y\})}{\sinh(1)} .
    \end{equation*}
    A generalization of \Cref{thm:sobolev-1-space} for inner products that include $\langle f, g \rangle_{L_2(0,1)}$ may be possible.
    In this case one would likely recover some type of piecewise \emph{exponential} interpolation as kernel interpolation.
    \item Most of the examples we reviewed in Sections~\ref{sec:examples-1} and~\ref{sec:examples-2} are connected to the Brownian motion. 
    Every kernel of the form~\eqref{eq:W1-kernel} can probably be viewed as a covariance kernel of a variant of the Brownian motion.
    \item Higher-order spline interpolants can be obtained from kernel interpolation based on the $m$ \emph{times integrated Brownian motion kernel}
    \begin{equation*}
        K_m(x, y) = \int_0^x \int_0^y K_{m-1}(t, s) \dif t \dif s = \int_0^1 \frac{(x-t)_+^m (y-t)_+^m}{(m!)^2} \dif t ,
    \end{equation*}
    where $K_0(x, y) = \min\{x, y\}$ and $(x)_+ = \max\{x, 0\}$~\cite[Ch.\@~1]{Wahba1990}.
    It is possible that $K_0$ can be replaced with any kernel of the form~\eqref{eq:W1-kernel}. 
    Expressing the corresponding inner product may prove challenging or complicated.
    If the inner product admits a convenient expression, the Green kernel interpretation and superconvergence theory in \Cref{sec:green-kernels,sec:superconvergence} may generalize.
\end{itemize}

\begin{acknowledgement}
TK was supported by the Research Council of Finland projects 359183  and 368086.
TK acknowledges the research environment provided by ELLIS Institute Finland.
GS is a member of INdAM-GNCS, and his work was partially supported by the project ``Perturbation problems and asymptotics for ellip-
tic differential equations: variational and potential theoretic method'' funded by the program ``NextGenerationEU'' and by MUR-PRIN, grant 2022SENJZ3.
\end{acknowledgement}

\bibliography{references}%
\bibliographystyle{abbrv}
\end{document}